  \documentclass{amsart}                
 \usepackage{graphicx}
  \usepackage[format=plain,font=small]{caption}
 
\usepackage{amssymb,amsmath,amsfonts}
  
 \usepackage[usenames]{color}

\def\b{\beta}
\def\z{\zeta}

\def\ng{\lambda}

\def\fo{{f_o}}
\def\N{{\mathbb N}}
\def\R{{\mathbb R}}
\def\C{{\mathbb C}}

\def\Q{{ Q}}
\def\Z{{\mathbb {Z}}}

\def\b{{\mathbf {b}}}
\def\x{{\mathbf {x}}}
\def\z{{\mathbf {z}}}

 \def\snc{{\rm sinc}}
\newcommand\what[1]{\widehat{#1}}

\newcommand\leng[2]{{\rm len}_{\phantom{}_{#2}}({#1})}

 \newcommand\pr{\mathop{{\text{\Huge{$\times$}}}}}
  \newcommand\crossp[2]{\mathop{\pr_{#1}^{#2}}}

\def\Ld{{L^2(\R)}} 
\def\Lu{{L^1(\R)}} 
 \def\Lp#1#2{{L_{#1}^{#2}}} 

  \def\ldCN {{{\ell}^2 (\Z;{\C}^N)}} 
\def \Tfi{{T^{}_{\Phi,\hskip 0.05em t_o}}}
\def\Tfia{{T^{*}_{\Phi,\hskip 0.05em t_o}}}

\def\Tfistar{{T^{}_{\Phi^*,\hskip 0.05em t_o}}}
\def\Tfistara{{T^{*}_{\Phi^*,\hskip 0.05em t_o}}}

\def\Efi{{E_{\Phi,\hskip 0.05em t_o}}}
\def\EfiL{{E_{\Phi_L,\hskip 0.05em t_o}}}
\def\Efistar{{E_{\Phi^*,\hskip 0.05em t_o}}}

\def\BL{{B_\omega}}
\def\Gfi{{G_{\Phi,\hskip 0.05em t_o}}}
\def\Gfia{{G_{\Phi,\Phi^*,\hskip 0.05em t_o}}}

   \newcommand\Gbfia{{G}^{}_{\Phi,\Phi^*,t_o}}

\def\Jfi{{J^{}_{\Phi,\hskip 0.05em t_o}}}

\def\Jfidual{{J^{}_{\Phi^{*},\hskip 0.05em t_o}}}

 \def\Jfistar{{J^{*}_{\Phi,\hskip 0.05em t_o}}}

   \newcommand\Apen[1]{{ {[\,\Jfi\,]}_{#1}}}
 
 \newcommand\Jbfidual{\mathbb{J}^{}_{\Phi^*,t_o}}
 \newcommand\Jbfistar{\mathbb{J}^{*}_{\Phi,t_o}} 
  
 \newcommand\Jbfi{\mathbb{J}_{\Phi,t_o}}
   
 \newcommand\cI {\mathop{ \mathcal I}}
  \newcommand\PI {{ {P}}_{\cI,h} }
 
\def\squareforqed{\hbox{\rlap{$\sqcap$}$\sqcup$}}
\def\qed{\ifmmode\squareforqed\else{\unskip\nobreak\hfil
\penalty50\hskip1em\null\nobreak\hfil\squareforqed
\parfillskip=0pt\finalhyphendemerits=0\endgraf}\fi} 
 \newtheorem{theorem}{Theorem}[section] 
\newtheorem{corollary}[theorem]{Corollary}
\newtheorem{lemma}[theorem]{Lemma}
\newtheorem{defn}[theorem]{Definition}
\newtheorem{proposition}[theorem]{Proposition}

\theoremstyle{remark}

\numberwithin{equation}{section}

\begin{document}
\title[Frames for band limited functions]
{Recovery of Missing Samples in Oversampling Formulas\\for Band Limited  Functions}

\subjclass[2000]{} 


 \author[Vincenza del Prete]
 {Vincenza Del Prete}

  \address{Dipartimento di Matematica\\
 Universit\`a di Genova, via Dodecaneso 35, 16146 Genova \\ Italia}

\begin{abstract}
In a previous paper  we  constructed  frames and oversampling formulas for band-limited functions, in the framework of  the theory of shift-invariant spaces. 
In this article we study the  problem of recovering  missing samples.
\par
We find  a sufficient  condition for the  recovery of   a  finite set of missing samples. The condition is expressed as a linear 
independence of the components of a vector $W$ over the space of trigonometric 
polynomials determined by the frequencies of the missing samples. We apply the theory to the derivative sampling of any order and we illustrate our results with a numerical experiment. 
 \end{abstract}

\maketitle
\setcounter{section}{0}
\section{Introduction} \label{s:Introduction}

  A band-limited signal is a function which belongs to the  space $\BL$  of functions in $\Ld$ whose Fourier transforms have support in $[-\omega,\omega]$.  Functions in  this   space can be represented  by  their Whittaker-Kotelnikov-Shannon series, which    is  the expansion   in terms of the  orthonormal basis of translates of the {\it sinc} function. The  coefficients
   of the expansion   are the  samples of the function  at  a uniform grid on $\R$,  with  ``density" $ {\omega}/{\pi}$ ({\it Nyquist} density). The sampling theory for band-limited functions stems from this  series representation and has been extended to   more general  expansions, as Riesz bases and frames, formed by the translates of one or more functions.  \par 
    Let $\Phi$ be a finite subset  of  $\BL$ and fix   a positive real number $t_o$.  If  the set of translates $\Efi=\{ \tau_{kt_o}\ \varphi,\,  \varphi\in \Phi,\,  k\in \Z \}$  is a frame for $\BL$ and the set  $\Phi^*$ of dual functions is known, then any    function   can be  reconstructed via  the syntesis formula (\ref{expansion}). The  elements of $\Phi$ are called {\it generators}.
 The coefficients of the expansion   are    samples of the convolution  of the function with the elements of   the dual family
   $\Efistar.$\par
  If the family is a Riesz basis,  the loss of even one sample prevents the reconstruction, since the elements of the family $\Efi$ are linearly  independent.  On  the contrary, if    the family $\Efi$  is  overcomplete,     recovery of a finite set of samples is possible   under suitable assumptions, because of the redundancy of the system. \par
 This paper deals with the problem of recovering missing samples of band-limited functions. We shall use the frame representation formulas constructed in \cite{DP} in the framework of the theory of frames for shift-invariant spaces. \par We recall that   $\BL$    is $t_o$-\emph{shift-invariant}   for any $t_o$, i.e. is invariant under all  translations  $\tau_{k t_o}, k \in \Z$, by integer multiples of   $t_o$.    A finite subset 
$\Phi $ of a $t_o$-shift-invariant space $S$     is called a set of generators  if $S$   is the closure of the space generated by the  family 
$\Efi.$ \
 Shift-invariant spaces   can have different sets of generators;
the smallest number of generators is called the length of the space. \par   The   structure of   shift-invariant spaces was first investigated  by
C. de Boor, R. DeVore   \cite{BDR}; successively    A. Ron and Z. Shen    introduced  their   Gramian analysis  and  characterized  sets of generators whose translates form    Bessel sequences, frames and Riesz bases  \cite{RS}.
 Their conditions are expressed in terms of the eigenvalues of the Gramian matrix.  More recently, building on their results, for the space $\BL$ we    have obtained  more explicit conditions   and  explicit formulas for the Fourier transforms of the dual generators     \cite{DP}; the  same  paper contains also   multi-channel   oversampling formulas for band-limited signals  and an application  to the    first and second derivative  oversampling formulas.  These are a generalization to frames   of the   classical   derivative sampling formulas, where  the family $\Efi$ is a Riesz basis (see \cite{Hi}). The coefficients of the  expansion are the  values of the function and  its derivatives at the sample points.  \par
In this   paper   we assume that $\Efi$ is a frame  for $\BL$  and we find  sufficient conditions for the recovery of a finite set of missing samples. This problem   has been investigated by  P.J.S.G Ferreira in  \cite{F}, where it is shown   that,  in the case of  a generalized Kramer sampling,   under suitable oversampling assumptions,   any finite set of  samples can be recovered from the others. Successively    D.M.S. Santos and Ferreira \cite{SF}   considered   the case of   a particular  two-channel  derivative  oversampling formula \cite{SF}.  The work of Santos and Ferreira has been generalized to arbitrary two-channels by Y. M. Hong and K. H. Kwon \cite{KK}.  \par
In this paper we  consider general frames $\Efi$ of $\BL$ where the number of elements of $\Phi$ is minimal, i.e. equal to the length of $\BL$ as $t_o$-shift invariant space, and   find sufficient conditions  for the recovery of  any finite subset of  missing samples. \par
 The paper is organized as follows. In Section 2  we   introduce the pre-Gramian and Gramian matrices  and recall
some  results, proved in \cite{DP},   that characterize sets of generators whose translates form   frames and Riesz bases, in terms of    the Fourier transforms of the generators.\par  
In  Section 3, we fix  the parameter $t_o$ and  we consider a family $\Phi=\{\varphi_1,\varphi_2 , \ldots   \varphi_N\}$
 of  $N$  generators, where $N$ is the length of ${\BL}$  as  a  $t_o$-shift-invariant space. If the family $\Efi$ is a frame and not a Riesz basis, we
 find   sufficient conditions     to reconstruct any finite set of missing samples (see Theorems \ref{t:recontructK} and \ref{t:recontructI}). This result makes use of    the particular  structure of the mixed-Gramian matrix    $\Gfia$ associated to  $\Phi$ and  to the dual generators $\Phi^*$.  Roughly speaking, the interval $[0, {2\pi}/{t_o}]$ is the union
 of two sets,  where $\Gfia$  is either the $N\times N$ identity  matrix    or $I-P_W$, where $P_W$ is the orthogonal  projection on a vector $W$ in $C^N,$ which  is the  cross product of $N-1$ translates of the Fourier  transform of the generators $\Phi$ (see Propositions \ref{p:gfifistarK} and \ref{p:gfifistar}). The recovery condition is expressed  as a linear independence of the components of the vector $W$ over the space of trigonometric polynomials  determined by the frequencies of the missing samples.  \\
  In  Section 4 we find     families  of  derivative frames of any order and     apply the results of   Section 3 to them, showing that it is possible  to recover any finite set of missing samples.     
 \section{Preliminaries}
\label{s: Preliminaries}
In this section  we establish notation and we  collect some  results  on frames for shift-invariant spaces. The Fourier transform of a  function  $f$ in 
  $  \Lu $ is
 $$ {\widehat f}(\xi)= \frac{1} {\sqrt{2\pi}}  \int f(t)\,e^{-it\xi}dt.$$     
 The convolution of two functions $f$ and $g$ is
 $ f\ast g (x)=\int f(x-y)g(y)\,dy,$ 
  so that $\widehat {f\ast g} = \sqrt{2\pi}  {\widehat f}\  {\widehat g}.$ 
Let $h$ be a positive real number; $    \Lp{h}{p}$ is the space of $h$-periodic functions on $\R$ such that $$\|f\|_{\Lp{h}{p}}=\Bigl(\frac{1}{h}\int_0^h|f(x)|^p dx\Bigr)^{1/p}<\infty.$$\par
In this paper    vectors in $\C^N$ are  to  be considered as column-vectors; however, to save space,   we shall write $x=(x_1,x_2, \dots, x_N)$ to denote the  column-vector whose components are $x_1, \dots,x_N$. 
 With the symbol $ \ldCN $ we shall denote the  space of square summable $\C^N$-valued sequences $c=(c(n))_{\Z}$. 
 \par\noindent      Let $H$ be a closed subspace of $\Ld.$   Given a subset $\Phi=\{ \varphi_j ,j=1,\dots, N\} $ of $H$ and a positive number $t_o$,  denote by $ \Efi$ the set    $$ \Efi =\{\tau_{nt_o}\varphi_j , \ n\in \Z, \ j=1,\dots ,N\};$$
  here   $\tau_a f(x)=f(x+a).$  
 The family $ \Efi $ is   a frame for     $H$   if the operator $\Tfi: \ldCN \rightarrow  H$ defined by
$\Tfi c=\sum_{j=1}^N \sum_{n\in \Z} c_{j}(n) \tau_{n t_o}\varphi_j$
is     continuous,  surjective and ${\rm{ran}}(\Tfi)$ is  closed. It is well known that     $ \Efi
$ is a   frame  for $H$  if  and  only if there exist two constants $0<A\le B$ such that\begin{equation}\nonumber
A\|f\|^2\leq \sum_{j=1}^{N}\sum_{n\in \Z}  |\langle f,  \tau_{n t_o}\varphi_j \rangle  |^2
\leq B \|f\|^2 
\hskip.6truecm \forall f\in  H.\end{equation} 
  The constants $A$ and $B$  are called frame bounds.
 Denote  by   $\Tfia: H \rightarrow \ldCN$ the    adjoint of $\Tfi.$           
The operator  $\Tfi\Tfia:H\rightarrow H$  is called {\it frame operator.} 
Denote  by  $\Phi^*$ the family  $\{\varphi_j^*, j=1,\dots ,N\},$   where
\begin{equation}\nonumber
 \varphi_j^*=(\Tfi\Tfia)^{-1}\varphi_j  \hskip1truecm     1\le j\le N.
 \end{equation}
If $\Efi$ is a frame for $H$ then    $\Efistar $ is also a frame,   called 
 the  {\it dual frame}, and  
$\Tfi\Tfistara=\Tfistar\Tfia =I.$
 Explicitly
  \begin{eqnarray}\label{expansion}
f= \sum_{j=1}^{N} \sum_{n\in \Z}  \langle f,  \tau_{n t_o}\varphi_j^* \rangle \tau_{n t_o}\varphi_j= \sum_{j=1}^{N} \sum_{n\in \Z}  \langle f,  \tau_{n t_o}\varphi_j \rangle \tau_{n t_o}\varphi_j^* \end{eqnarray}  
$  \forall f\in H.$
Note that 
\begin{equation}\label{innersamples}
\langle f,\tau_{nt_o} g\rangle=f\ast \tilde{g}(-nt_o) \qquad \forall n\in\Z
\end{equation}
where $\tilde{g}$ denotes the function $\tilde{g}=\overline{g}(-t)$. Thus the coefficients of the expansion are the samples  of the function $f\ast\tilde{g}$ in $-nt_o$.
  The elements of  $\Phi$ are called {\it generators} and the elements of $\Phi^*$ {\it dual generators.} 
 If the family    $ \Efi
$ is a frame for $H$ and the operator $\Tfi$ is injective, then $ \Efi
 $ is called a  {\it  Riesz basis}. 
\par
 In what follows    $t_o$ is a positive parameter.  To simplify notation, throughout the paper we shall set  $$h=\frac{2\pi}{t_o}.$$
\par A subspace $S$ of $\Ld$   is $t_o$-shift-invariant if it is invariant under all  translations by  a multiple of $t_o.$  
The following bracket product  plays an important role in  Ron and Shen's analysis of shift-invariant spaces. For  $f$ and $g \in\Ld,$  define 
\begin{equation}\nonumber
[f,g]=h\sum_{j\in \Z} f(\cdot+jh ) \overline{g}(\cdot+jh).
\end{equation}
Note that $[f,g]$ is in $\Lp{h}{1}$ and $\|[f,f]\|_{\Lp{h}{1}}=\|f\|_2^2.$ 	 The  Fourier
coefficients of $ [\hat{f},\hat{g}]$ are given by
\begin{align}\nonumber
[\hat{f},\hat{g}]{\,}^{\widehat {\,}} (n)&=\int_0^{h}
\sum_{j }\tau_{jh}(\hat{f}\, \overline{\hat{g}})(x) e^{-2\pi i n \frac{x}{h}}dx=
\langle f, \tau_{n t_o}  g\rangle\\ &=f\ast \tilde{g}(-nt_o)\quad  \qquad \forall n\in \Z.\label{braketcoeff} 
\end{align}
If $S$ is a $t_o$-shift-invariant    space and   there exists     a finite family  $\Phi$ such that   $S $ is the closed linear span of $\Efi$, then we say that $S$ is finitely generated. Riesz bases for   finitely generated 
 shift-invariant spaces have been studied by various authors. In \cite{BDR} the authors  give a characterization of such bases.     A   characterization of frames and tight frames    also for countable sets $\Phi$ has been given by Ron and Shen in \cite{RS}.   The principal notions of their theory are the {\it    pre-Gramian}, the {\it Gramian} and the {\it dual Gramian} matrices.
The {\it pre-Gramian} $\Jfi$ is the $h$-periodic function mapping $\R$ to the space of $\infty\times N$-matrices defined on $[0,h]$  by
\begin{equation}\label{pregramian}
\bigl(\Jfi \bigr)_{j\ell}(x)=\sqrt{h}\  \what{\varphi_\ell}(x+jh),\hskip1truecm j\in \Z, \,  \ell=1,\dots,N.
\end{equation}
The pre-Gramian $\Jfi$ should not be confused with the  matrix-valued function  whose entries  are $\sqrt{h}\, \what{\varphi_\ell}(x+jh),$ for all $x\in\R,$ which is not periodic.  
  Denote by $\Jfistar$ the adjoint of $\Jfi.$ 
 The   $N\times N$ Gramian  matrix  $\Gfia=\Jfistar \Jfidual$ plays a crucial role in the recovery of   missing samples; its elements 
 are the $h$-periodic functions  
  \begin{equation}\label{gramianstar}
\bigl(\Gfia \bigr)_{j\ell}=[\what{\varphi^*_{\ell}}  ,\what{\varphi_j}].
\end{equation} 
    We  set $$\ell=\left [ \frac{\omega}{h} \right ]+1,$$  where  $[a]$ denotes the greatest integer less than $a.$ 
In \cite[Corollary 2.3]{DP} we found    the length of $\BL$,  as  $t_o$-shift-invariant space:   \begin{equation}\label{lenBL}
 \leng{\BL}{t_o}=\begin{cases}
 2\ell   &   {\rm {if}}\ \frac{\omega}{\ell}\le h<\frac{\omega}{\ell-\frac{1}{2}}, \\
 2\ell-1   & {\rm {if}} \   \frac{\omega}{\ell-\frac{1}{2}}\le h< \frac{\omega}{\ell-1}.\end{cases}
 \end{equation} 
We also  gave   necessary and sufficient conditions for $\Efi$ to be a Riesz basis or a frame for $\BL$ \cite[Theorems~3.6, 3.7]{DP}.  The result is  based on  the analysis  of the  structure of the matrix $\Jfi$;   we  restate it below in Theorems \ref{maingen} and \ref{maingendue} for the sake of the reader. \\
Strictly speaking, the pre-Gramian $\Jfi$ is an infinite matrix. However,   in  \cite[Lemma 3.3]{DP} it has been shown that if  $\Phi$ is a set of generators of  $\BL$, then 
all but a finite number of the rows of $J_{\Phi,to}$ vanish; hence we may identify it with a finite matrix.  Indeed, consider separately the two cases $  \omega/\ell \le h<  \omega/(\ell- {1}/{2})$
and $ {\omega}/(\ell- {1}/{2})\le h<  {\omega}/({\ell-1}).$\par    Assume first that $ {\omega}/{\ell}\le h<  {\omega}/(\ell- {1}/{2});$ then  $0\le -\omega+\ell h<\omega-(\ell-1)h<h$.  
 We denote by
$I_{-},I,I_{+}$ the intervals defined by
  \begin{equation}\nonumber
I_{-}=(0,-\omega+\ell h),\quad  I= (-\omega+\ell h,\omega-(\ell-1)h), \quad  I_+=(\omega-(\ell-1)h,h). 
\end{equation} 
 By  
 (\ref{lenBL})   $\leng{\BL}{t_o}=2\ell$;  if    $\Phi=\{\varphi_j:1\le j\le 2\ell\}$ is  a subset of $\BL$  of cardinality $2\ell$, by   \cite[Lemma 3.3]{DP} 
 all the rows of the matrix $\Jfi $ vanish, except possibly the rows $\big(\tau_{j h}\what{\varphi}_1,\tau_{j h}\what{\varphi}_2,\dots,\tau_{j h}\what{\varphi}_{2\ell}\big)$, $-\ell\le j\le\ell-1$. 
Thus we identify the infinite matrices $\Jfi$,   $\Jfidual$,  their adjoints and the matrices $\Gfi$, $\Gfia$ with their $2\ell\times2\ell$ submatrices corresponding to these rows.
The $i$-th column  of $\Jfi,$   $1\le i\le 2\ell$ is\begin{equation}\label{colonne:0}
\sqrt{h}\begin{bmatrix} 0\\
\tau_{-(\ell-1)h}\what{\varphi_i}\\
\vdots \\
\what{\varphi_i}\\
\\ \vdots \\
\tau_{(\ell-2)h}\what{\varphi_i}\\
\tau_{(\ell-1)h}\what{\varphi_i}\\
\end{bmatrix}{\rm{in}}\  I_{-}\hskip0.5truecm
\sqrt{h}\begin{bmatrix} \tau_{-\ell h}\what{\varphi_i}\\
\tau_{-(\ell-1)h}\what{\varphi_i}\\
\vdots \\
\what{\varphi_i}\\
\\ \vdots \\
\tau_{(\ell-2)h}\what{\varphi_i}\\
\tau_{(\ell-1)h}\what{\varphi_i}\\
\end{bmatrix}\ {\rm{in}}\  I\hskip0.5truecm
\sqrt{h}\begin{bmatrix} \tau_{-\ell h}\what{\varphi_i}\\
\tau_{-(\ell-1)h}\what{\varphi_i}\\
\vdots \\
\what{\varphi_i}\\
\\ \vdots \\
\tau_{(\ell-2)h}\what{\varphi_i}\\
0
\end{bmatrix}\  {\rm{in}}\  I_{+} 
\end{equation}
 
\vskip.5truecm\noindent
(see  Lemma 3.3 in \cite{DP}).
The same  formulas hold for the matrix $\Jfidual$, with $\varphi$ replaced by $\varphi^*.$  We   note that the matrices  $ { \Gfi}$ and ${\Jfi}$ have the same rank.
Let $A$  be a $n\times m$ matrix with complex entries, $n\le m$; we shall denote by $\|A \|$ the norm of $A$ as linear operator from $\C^m$ to $\C^n$ and by $[A]_n$ the sum of the squares of the absolute values of the minors of order $n$ of $A.$
   \begin{theorem}\label{maingen}\cite[Theorem~3.6]{DP} Suppose that \
   $ {\omega}/{\ell}\le h< {\omega}/({\ell- {1}/{2}}).$ 
     Let $\Phi= \{\varphi_j:1\le j\le2\ell\} $ be a subset of $\BL$.    Then $\Efi$ is a frame for $\BL$ if and only if  there exist positive constants $\delta,\gamma,\sigma$ and $\eta$ 
such that 
    \begin{equation}\label{maingen:1}
 \delta\le \sum_{j=1}^{2\ell}|\what{\varphi}_j|^2\le\gamma
 \hskip1truecm a.\,e.\  in  \quad (-\omega,\omega),
\end{equation}
  \begin{equation}\label{maingen:2}
  \Apen{2\ell-1} \ge \sigma  \hskip1truecm a.e.\  in \quad  I_{-}\cup I_{+}\hskip 0.04em, 
\end{equation}
   \begin{equation}\label{maingen:3}
  |\det\,{\Jfi}| \ge \eta  \hskip1truecm a.e.\  in \quad  I.\end{equation}
If    
$ h= {\omega}/{\ell}$  the intervals $I_{-}$ and $I_{+}$ are empty. In this case  $\Efi$ is a Riesz basis  for $\BL$ if and only if    conditions (\ref{maingen:1})   and 
 (\ref{maingen:3}) hold.  
\end{theorem}
Next we consider the case 
$ \omega/ ({\ell- {1}/{2}})\le h<  {\omega}/({\ell-1})$. 
Then  $0<\omega- (\ell-1)h  \le -\omega+\ell h<h;$ in this case  we denote by  $K_{-},K,K_{+}$ the intervals defined by
\begin{equation}\label{K}
K_{-}=(0,\omega-(\ell-1) h),\quad  K=  (\omega-(\ell-1) h,-\omega+\ell h),\quad  K_+=(-\omega+\ell h,h).
\end{equation}
\vskip.05 truecm
By   
  (\ref{lenBL})   $\leng{\BL}{t_o}=2\ell-1$. 
Let    $\Phi=\{\varphi_j:1\le j\le 2\ell-1\}$ be a subset of $\BL$  of cardinality $2\ell-1$. By \cite[Lemma 3.3]{DP} 
all the rows of the matrix $\Jfi,$ except possibly $\big(\tau_{j h}\what{\varphi}_1,\tau_{j h}\what{\varphi}_2,\dots,\tau_{j h}\what{\varphi}_{2\ell}\big)$, $-\ell\le j\le\ell-1,$ vanish. 
Thus we identify the infinite matrices $\Jfi$,   $\Jfidual$, their adjoints and the matrices $\Gfi$, $\Gfia$ with their $2\ell-1\times2\ell-1$ submatrices corresponding to these rows. 
The $i$-th column  of $\Jfi,$   $1\le i\le  2 \ell-1$ is 
\begin{equation}\label{colonne:1}
\sqrt{h}\begin{bmatrix} 
\tau_{-(\ell-1)h}\what{\varphi_i}\\
\vdots \\
\what{\varphi_i}\\
\\ \vdots \\
\tau_{(\ell-2)h}\what{\varphi_i}\\
 \tau_{(\ell-1)h}\what{\varphi_i}\\
\end{bmatrix}{\rm{in}}\  K_{-}\hskip0.3truecm
\sqrt{h}\begin{bmatrix} 
\tau_{-(\ell-1)h}\what{\varphi_i}\\
\vdots \\
\what{\varphi_i}\\
\\ \vdots \\
\tau_{(\ell-2)h}\what{\varphi_i}\\
 0\\
\end{bmatrix}\ {\rm{in}}\  K\hskip0.3truecm
\sqrt{h}\begin{bmatrix} \tau_{-\ell h}\what{\varphi_i}\\
\tau_{-(\ell-1)h}\what{\varphi_i}\\
\vdots \\
\what{\varphi_i}\\
\\ \vdots \\
\tau_{(\ell-2)h}\what{\varphi_i} \end{bmatrix}\  {\rm{in}}\  K_{+}.
\end{equation}
\vskip1truecm 
\par
\begin{theorem}\label{maingendue}  \cite[Theorem~3.7]{DP} Suppose that 
$ {\omega}/({\ell-{1}/{2}})\le h<  {\omega}/({\ell-1}), $ 
 $\ell\not=1.$   Let  $\Phi=\{\varphi_j:$$1\le j<2\ell-1\}$ be a subset of $  \BL$.
Then $\Efi$ is a frame for $\BL$ if and only if  there exist positive constants $\delta, \gamma, \sigma$ and $\eta$ such that
    \begin{equation}\label{maingendue:1}
\delta\le \sum_{j=1}^{2\ell-1}|\what{\varphi}_j|^2\le\gamma
 \hskip1truecm a.e. \   in\  (-\omega,\omega),
\end{equation}
  \begin{equation}\label{maingendue:2}
\Apen{2\ell-2}\ge \sigma  \hskip1truecm a.e.\    in \ K, \end{equation}
   \begin{equation}\label{maingendue:3}
|\det\, \Jfi|\ge \eta  \hskip1truecm a.e.\   in \   K_{-}\cup K_{+}.
\end{equation}
If $h={\omega}/(\ell-\frac{1}{2})$  the interval $K$ is empty. In this case  $\Efi$ is a Riesz basis if and only if conditions  (\ref{maingendue:1}) and (\ref{maingendue:3}) hold. 
\end{theorem}
 \section{Recovery of missing samples}
\label{s: Recovery}
In  this  section  we   consider  the problem of    reconstructing a 
band-limited function   when a finite set  of samples is missing. Let $t_o$ be a  positive real number  and $L$ the length of $\BL$ as  
$t_o$-shift-invariant space.  Let  $\Phi=\{\varphi_1,\dots,\varphi_L\}$ be a   set of functions in $\BL$  such that $\Efi$ is a frame for $\BL.$ Then any    function in $\BL$ can be  reconstructed via  the syntesis formula (\ref{expansion}). We recall  that  the coefficients of the expansion    are   the   samples of  the functions $f_j=f\ast\tilde{\varphi}_j$ at  integer multiples of $t_o$  and we may rewrite (\ref{expansion}) as  \begin{eqnarray}\label{expansion2}
f=   \sum_{j=1}^{L} \sum_{n\in \Z} f_j (nt_o)   \tau_{-n t_o}\varphi_j^*  \qquad \forall f\in \BL.\end{eqnarray}  
  The recovery of lost samples in band-limited signals has  already been investigated   by Ferreira for the classical  (one-channel) Shannon formula \cite{F}, and  by Santos and Ferreira for  a particular   two-channel derivative sampling \cite{SF}. In the  latter article,  the authors show that a finite number of missing samples of the function or  its derivative can be recovered.  In \cite{KK}     Hong and   Kwon have generalized these results to a two-channel  sampling formula, finding  sufficient conditions for the    recovery of  missing samples.
In both papers  \cite{SF, KK}, the authors  work with particular reconstruction formulas, obtained  by projecting   Riesz  basis generators  of  $\BL$   and their  duals    into   the space  $B_{\omega_a} $ with $\omega_a<\omega$. With this technique  the   projected family is a   frame; note  that     projecting  the  dual of a Riesz basis  does not yield the  canonical  dual. In particular,  the coefficients of the expansion of a function computed with respect to the projected duals   are not minimal in  the least square norm.   \par 
 We find  sufficient  conditions  such that, if $\Efi$ is a frame (and not a Riesz basis), any  finite set of missing samples may  be recovered (see Theorems \ref{t:recontructK} and  \ref{t:recontructI} below).\\ 
 In the   recovery  of missing    samples     the structure of the mixed Gramian matrix  $\Gfia= \Jfistar \Jfidual$ plays a crucial role. 
 Propositions     \ref{p:gfifistarK} and \ref{p:gfifistar} below show   that,   a.e. in the interval $[0,h]$,   the matrix $I-\Gfia$    is either the identity    or a   projection  on a subspace of codimension one. To prove them,  we    recall some matrix identities   obtained in \cite{DP}. 
 By formula   (4.2)   in \cite{DP}
  \begin{equation}\label{generators:1}
 \Jfi=\Jfi\Jfistar\, \Jfidual=\Jfi  \Gfia.
 \end{equation} 
 Under the assumptions of Theorems \ref{maingen}  or   \ref{maingendue}, the interval $[0,h]$ is the disjoint union of three intervals where the pre-Gramian is either invertible or has rank $L-1$.  Where the pre-Gramian is invertible   one 
 can solve for $\Jfidual$ in (\ref{generators:1}), obtaining  that 
\begin{equation}\label{Gei:1}
\Jfidual=(\Jfistar)^{-1}.\end{equation}
  In  the intervals where  the pre-Gramian has rank $L-1,$  either its first or  last row   vanishes; this happens   in the intervals $I_{-}\cup I_{+}$  in    formula (\ref{colonne:0}) and in the interval $K$ in (\ref{colonne:1}). We shall denote by $\Jbfi$ and $\Jbfidual$ the $(L-1)\times L\, $  submatrices of $\Jfi$ and $\Jfidual$ obtained by deleting the vanishing row.  In this case $   \Gfia=\Jbfistar\Jbfidual $. Since $ \Jbfidual$ 
is the Moore-Penrose inverse $(\Jbfistar)^\dagger$  of $\Jbfistar$,  we have that
\begin{equation}\label{Gei:3}
 \Gfia=\Jbfistar (\Jbfistar)^\dagger . 
\end{equation}
 We refer the reader to \cite{BIG} for the definition and the properties of  the Moore-Penrose  inverse of a  matrix.
  \\ \indent
 We shall denote   by $W$ the cross product of the rows of the matrix $\  h ^{-1/2}\Jbfi$. 
 We observe that, if 
 \ $ {\omega}/({\ell-{1}/{2}})< h< {\omega}/({\ell-1}),$ 
 $\ell\not=1$, then 
 \begin{equation}\nonumber
 W=\crossp{j=-\ell+1}{\ell-2}\tau_{jh}\what{\Phi}\qquad {\rm a.e.\  in\ } K;\end{equation}
while, if 
  $ {\omega}/{\ell}< h<  {\omega}/(\ell-  1/2) ,$ 
  then 
 \begin{equation} \nonumber W = \crossp{j=-\ell+1}{\ell-1}\tau_{jh}\what{\Phi}\quad {\rm a.e.\  in\ }\ I_{-} \qquad W =\displaystyle{\crossp{j=-\ell}{\ell-2}}\tau_{jh}\what{\Phi} \quad{\rm a.e.\  in\ }\  I_{+}.\end{equation}  
 Observe that, since $W$ is orthogonal to the rows of $\Jbfi$ and ${\rm dim\ Ker\, }\Jbfi=1$, then    ${\rm   Ker\, }\Jbfi={\rm span} (W).$   
\par  We   denote by $I_n$ the    $n\times n$  identity matrix. 
 Given a vector $v\in \C^n$, we  denote by    $P_{v}$ the orthogonal projection on $v$   in $\C^{n}$.
 \begin{proposition}\label{p:gfifistarK} 
  Let ${\omega}/(\ell- 1/2)< h< {\omega}/(\ell-1),$ 
   $\ell\not=1$ and let 
 $\Phi=$\penalty-10000$\{\varphi_j,$ $ 1\le j\le2\ell-1\}$ be a subset of  $  \BL.$  
If   $\Efi$ is a frame for $\BL$
 then      
 \begin{align} 
 \Gfia&=I_{2\ell-1}\hskip 1truecm a.e.\, \textrm{in} \
   K_-\cup K_+, \label{p:GI}     \\
   {}\nonumber\\ 
   \Gbfia&=I_{2\ell-1}- P_{{}_W}
    \hskip 1truecm a.e.\, \textrm{in} \ K.\label{p:GIP}       \end{align}  
 \end{proposition} 
\begin{proof}
  By (\ref{maingendue:3}), the pre-Gramian $\Jfistar$ is invertible in the intervals $ K_-$ and  $K_+$; hence  formula (\ref{p:GI}) follows   from (\ref{generators:1}).  Consider now the interval $K$. By (\ref{Gei:3}) and the properties of the Moore-Penrose inverse, $  \Gfia$ is the orthogonal projection on $Ran(\Jbfistar)$ in $\C^{2\ell-1}$.  On the other hand, since $\Jbfi$ has maximum rank,
$$ span(W)=Ker(\Jbfi)=Ran(\Jbfistar)^\perp.$$ 
Hence the projection on $Ran(\Jbfistar)$ is $I_{2\ell-2}-P_{\,W}$. This  proves (\ref{p:GIP}). \\
\end{proof}
\noindent A similar argument yields
  \begin{proposition}\label{p:gfifistar}
 Let 
 ${\omega}/{\ell}< h< {\omega}/(\ell- {1}/{2})$   and let 
 $\Phi=\{\varphi_j, 1\le j\le2\ell\}$ be  a subset of $ \BL.$
If   $\Efi$ is a frame for $\BL$,  then      
  \begin{align}  
   \Gfia&=I_{2\ell}\hskip 1truecm  a.e.\,\textrm{in}\     I \label{p:GII}
 \\{}\nonumber\\ \Gbfia&= I_{2\ell}-P_{{}_{W}}
  \hskip 1truecm  a.e.\,\textrm{in}\     I_-\cup I_+ .
 \label{p:GIPI} 
  \end{align} 
\end{proposition} 
\par 
Let  $  \cI=\{ \ell_1,\ell_2,\dots,\ell_N\} $  be a  given set of integers. 
We give    sufficient conditions to reconstruct 
the    missing samples  
\begin{equation}\label{coeff}  f_j( nt_o) 
  \hskip 1truecm   n\in\cI,\qquad   1\le j\le L.
\end{equation}
 Theorem  \ref{t:recontructK} deals with the case   
 $ {\omega}/(\ell- 1/2)< h<  {\omega}/(\ell-1), $
   i.e. $L=2\ell-1$, while Theorem \ref{t:recontructI} deals with the case 
    $ {\omega}/ \ell < h<  {\omega}/(\ell-{1}/{2}),$
     i.e. $L=2\ell.$\\
 The sufficient conditions are expressed   as the linear independence of the components  of the vector $W$ over  the set  $\PI$   of trigonometric polynomials   of the form 
$$p(t)=\sum_{j=1}^N x_j e^{-\frac{2\pi}{h} i \ell_j t}\hskip 1truecm   x_j\in \C.$$
\begin{defn}  A finite family $\{F_j, j=1,\dots,m\},$ $m\in\N,$ of  functions   is $\PI$- linearly  dependent on a set $\Omega\subset[0,h]$ if there exist non-zero trigonometric polynomials $p_1,p_2,\dots p_m,$  in $ \PI$ such that 
\begin{equation}\label{d:lineari}
\sum_{k=1}^m\,p_j(t)\, F_j(t)=0 \hskip1truecm   for  \ a.e.\, t\in \Omega
\end{equation}
\end{defn}
 If    ${\x}=(x_1,x_2,\dots,x_N) $  is a vector  in $\C^N$, we shall  denote by $\what{\x}$ the trigonometric polynomial in $\PI$
\begin{equation}
\what{\x}(t)=\sum_{j=1}^N x_j e^{-\frac{2\pi}{h} i \ell_j   t}
 \hskip 1truecm t\in \R.
\end{equation}
If $ X=\big(\x_1, \x_2,\ldots,\x_L \big) $  with $\x_j\in \C^N$ for $j=1,\ldots,L$, we write
 \begin{equation}\label{Xhat2}   \widehat{X}(t) =\Big(    \widehat{\x_1}(t),   \widehat{\x_2}(t),\ldots, \widehat{\x_L} (t)\Big) \hskip 1truecm t\in \R.
\end{equation}
Thus $X\mapsto \widehat{X}(t)$ maps vectors in $\C^{N\times L}$  to vectors in $\C^L$ for a.e. $t\in \R$.
 \begin{theorem} \label{t:recontructK} Suppose that the assumptions   of Theorem \ref{maingendue} are  satisfied with 
 $h\not= {\omega}/(\ell-1/2)$. 
 Let  
$ \cI \subset \Z$ be an assigned  finite set. If  the components of the vector $W$ are 
$\PI$-independent on $K$, then for every $f\in \BL$
 it is possible to recover the samples 
$f_j(nt_o)$,
 $n\in {\cI},$ $ 1\le j\le 2\ell-1$.  \end{theorem}
 \begin{proof}
 Let  $  \cI=\{\ell_1,\ell_2,\dots,\ell_N\}$.  By  convolving  with   $\widetilde{\varphi}_k $ both sides of   expansion formula (\ref{expansion2}),  where  $L=2\ell-1$,   
 and evaluating  in $\ell_m t_o$,  we get
  \begin{equation} \nonumber
f_k(\ell_m t_o)
 =   \sum_{j=1}^{2\ell-1} \sum_{n\in \Z} 
  \,  f_j(nt_o)
  \, \big( \varphi_j^* \ast \widetilde {\varphi}_k \big)
 (\ell_m t_o-nt_o)   \end{equation}  
 $k=1,\dots,2\ell-1$, $m=1\dots,N$.
Next we isolate in the left hand  side the  terms containing  the unknown samples 
       \begin{align} \nonumber 
f_k(\ell_m t_o)-  \, 
    & \sum_{j=1}^{2\ell-1}
     \sum_{p=1 }^{N} f_j(\ell_p t_o)\ 
    \big( \varphi_j^* \ast \widetilde{ \varphi}_k \big)(\ell_m t_o-\ell_p t_o)  \nonumber \\ {} \label{t:system} 
  \\=&   \sum_{j=1}^{2\ell-1} 
   \sum_{n\not\in  \cI }  f_j(nt_o)\ 
   \big( \varphi_j^* \ast \widetilde {\varphi}_k)\big)(\ell_m t_o-nt_o)   \nonumber 
   \end{align}  
 $k=1,\dots,2\ell-1$, $1\le m\le N.$ 
This is a system of $N(2\ell-1)$ equations in the $N(2\ell-1)$ unknowns  
 $f_k(\ell_m t_o).$
To write it in a more compact form, we denote  the unknowns by  $x_j(m)$  and  the right hand side of equation (\ref{t:system}) by $b_k(m)$, i.e.
 \begin{align}\nonumber
x_k(m)&=  f_k(\ell_m t_o)\\
  b_k(m)&=  \sum_{j=1}^{2\ell-1} 
   \sum_{n\not\in  \cI }  f_j(nt_o)
   \,  \big( \varphi_j^* \ast \widetilde{\varphi}_k\big)(\ell_m t_o-nt_o),\nonumber 
\end{align}
for $1\le k\le 2\ell-1,$ $1\le m\le N.$ 
 Next, we write $$\x_k=\big(x_k(1),\ldots,x_k(N)\big),\qquad \b_k=\big(b_k(1),\ldots,b_k(N)\big)$$ and 
 $$X=(\x_1,\ldots,\x_{2\ell-1}),\qquad B=(\b_1,\ldots,\b_{2\ell-1}).$$
We introduce the block matrix   
\begin{equation}\label{t:matrix}
S=\begin{bmatrix} S_{11}&S_{12}&\cdots&S_{1(2\ell-1)} \\ S_{21}&S_{22}&\cdots&S_{2(2\ell-1)} \\ \vdots&\vdots&\ddots&\vdots \\ S_{(2\ell-1)1}&S_{(2\ell-1)2}&\cdots& S_{(2\ell-1)(2\ell-1)} \end{bmatrix}
\end{equation}
where the  $S_{k j},$ $  1\le j,k\le 2\ell-1,$  are  the submatrices whose entries are     
\begin{equation}\label{t:block}
S_{k j}(m,p)=  \big( \varphi_j^* \ast \widetilde{\varphi}_k\big)(\ell_m t_o-\ell_p t_o)
\hskip1truecm    m,p=1,\dots,N.
 \end{equation} 
Thus    the equations in (\ref{t:system}) can be written
\begin{equation}\label{t:matrixeq}(I-S){X} =B 
\end{equation}
where    $X$ and $B$ are  vectors in $\C^{N(2\ell-1)}$ and $I$ is the $N(2\ell-1)\times N(2\ell-1)$ identity matrix.
Thus the missing samples can be recovered if equation (\ref{t:matrixeq})  can be solved  for all $B$, and this happens  if and only if   $1$ is not an eigenvalue of $S$. \\
We shall prove that  if 1 is an eigenvalue of $S$, then  the components of the vector $W$  are $\PI$-dependent on $K$; thereby contradicting the assumptions. Consider the quadratic form $ {X}^*X-X^*SX$, where   $\ {}^{*}$ denotes    conjugate transpose.
We claim  that  for all $X\in \C^{N(2\ell-1)}$ 
\begin{equation}\label{t:integral}
{X}^*X-X^*SX=
 \frac{1}{h}\int_{K}\Big\|  \Big(  I_{2\ell-1}\,-\,  \Gbfia(t)\Big)
  {\widehat{X}} (t)
 \Big\|_{\C^{2\ell-1}}^2 
 \ dt,
\end{equation}
 where $ {\widehat{X}}$ is given by (\ref{Xhat2})  with $L=2\ell-1.$ 
Indeed,  by   (\ref{braketcoeff}), the  $(m,p) $-entry  of the matrix $S_{k j}$ in (\ref{t:block}) is
$$ S_{k j}(m,p)
= \frac{1}{h}\int_{0}^{h} [\widehat{ \varphi}_j^*,  \widehat{\varphi} _k](t)\,e^{i(\ell_m-\ell_p)t_o \,t}\, dt \hskip1truecm m,p=1,\ldots,N.
$$
  Hence  
\begin{align} 
 {X}^*SX  =  &\ \sum_{j,k=1}^{2\ell-1} \sum_{m,p=1}^N\, \overline{x}_j(m)\, S_{j,k}(m,p)\, {x}_k(p)\nonumber\\ 
=&   \frac{1}{h}
\sum_{j,k=1}^{2\ell-1} \sum_{m,p=1}^N\, \overline{x}_j(m) \,\int_0^h 
 [\widehat{ \varphi}_j^*,  \widehat{\varphi} _k](t)\,e^{i(\ell_m-\ell_p)t_o \,t}\, dt \ {x}_k(p)\  \nonumber
 \\ =&
  \frac{1}{h}\int_0^h 
 \sum_{j,k=1}^{2\ell-1} \sum_{m=1}^N\, \overline{x}_j(m)e^{i\ell_m t_o \,t}  \,
   [\widehat{ \varphi}_j^*,  \widehat{\varphi} _k](t)\, \sum_{p=1}^N\,\ {x}_k(p)\ e^{-i\ell_p t_o \,t} dt
    \nonumber\\
    =& 
    \frac{1}{h}\int_{0}^h   
\widehat{X}^*(t) \, \Gfia(t) {\widehat{X}}(t)\ dt.\nonumber
\end{align}
The last identity follows by (\ref{gramianstar}).
Thus, by  the Parseval identity 
$$ {X}^* X=\frac{1}{h}\int_0^h \,  {\widehat{X}}^* (t)
{\widehat{X}} (t)\, dt$$
and the fact  that $\Gfia(t)$ is an orthogonal projection  for  almost every  $t\in[0,h],$
 we obtain
\begin{align}\nonumber
{X}^* X-{X}^*S\,X=&
 \frac{1}{h}\int_{0}^h    {\widehat{X}}^* (t)\Big(I_{2\ell-1}- \Gfia(t) \Big)\, {\widehat{X}}(t) \ dt\nonumber \\
 = &
 \frac{1}{h}\int_{0}^h   {\widehat{X}^*} (t)\Big(I_{2\ell-1}- \Gfia(t) \Big)^2\, {\widehat{X}}(t) \ dt \nonumber \\
 =&  \frac{1}{h}\int_{0}^h \Big\|
  \big(I_{2\ell-1}- \Gfia(t) \big) \widehat{X}(t)
 \Big\|_{\C^{2\ell-1}}^2\, dt.\nonumber
\end{align}
Formula (\ref{t:integral}) follows by observing that,  by  (\ref{p:GI}),  the integrand is zero in \penalty-100000$K_{-}\cup K_{+}=[0,h]\setminus K$. \\
  Assume    that $1$ is an eigenvalue of the matrix $S$. Then, by (\ref{t:integral}), there is a vector   $X^o=(\x_1^o,\ldots,\x^o_{2\ell-1})$ in  $\C^{N(2\ell-1)}$, $ X^o\not=0$, such that
\begin{equation}\label{t:final}
 \big( I_{2\ell-1}\,-\,  \Gbfia(t)\big) {\widehat{X}}^o (t)=0 \hskip1truecm    {\rm for \ a.e.\ } t   \in  \, K.
\end{equation}
Since $I_{2\ell-1}\,-\,  \Gbfia=P_{{}_W} $ a.e.  in $K,$  by  
 Proposition \ref{p:gfifistarK}, the vectors  
   $ {\widehat{X}}^o (t)$ and $ {W(t)}$ are orthogonal for a.e. $t$ in $K$,  i.e.   $$\sum_{j=1}^{2\ell-1} \overline{\what{\x_j^o}}(t)W_j(t)  =0 \hskip 1truecm {\rm  for \ a.e.} \ t \ {\rm    in}\ K.$$
   Since     $\what{\x_j ^o} (t)\in \PI$ , $1\le j\le 2\ell-1$,  this implies  that the components of $W$ are  $\PI$-dependent.\end{proof}
 \noindent {\it Remark.} 
 If the matrix $S$ is Hermitian, the condition in the previous theorem is also necessary. Indeed,  assume that  the components of $W$ are $\PI$-linearly  dependent on the interval $K.$
Then   
there exists a non-null vector $X^o=(\x^o_j,\x^o_2,\dots \x^o_{(2\ell-1)})$ in  $ \C^{N(2\ell-1)}$   such that  ${\widehat{X}}^o$ is orthogonal to $W$ in $K$. Thus $P_{{}_W} \what{X}^o=0$; 
since by (\ref{p:GIP}), $P_{{}_W}=I_{2\ell-1}\,-\,  \Gbfia$  the identities
(\ref{t:integral}) and (\ref{t:final})  imply that 
 $ {X^o}^*X^o-{X^o}^*SX^o=0.$ Hence, since the matrix $I-S$ is Hermitian and  the associated quadratic form is positive semidefinite,
  $ X^o - SX^o =0$, i.e. 1 is an eigenvalue of $S$.
 \begin{theorem} \label{t:recontructI}
 Suppose that  the assumptions   of Theorem \ref{maingen}  are satisfied with $h\not= {\omega}/{\ell}$ and let 
 $  \cI  \subset \Z$ be a finite assigned set.  If  the components of the vector $W$ are 
$\PI$-independent on $I_{-}\cup I_{+}$, then it is possible to recover the samples $f_j(nt_o)$,
  \penalty-10000
 $n\in {\cI},\,  1\le j\le 2\ell. $
  \end{theorem}
 \begin{proof}
 Let $  \cI=\{\ell_1,\ell_2,\dots,\ell_N\}$.  As in the proof of the previous theorem, we write the system as in (\ref{t:matrixeq}),  where 
 the block matrix $S$ is now   
\begin{equation}\nonumber
S=\begin{bmatrix} S_{1\, 1}&S_{1\,2}&\cdots&S_{1\,2\ell
} \\ S_{2\,1}&S_{2\,2}&\cdots&S_{2\,2\ell
} \\ \vdots&\vdots&\ddots&\vdots \\ S_{2\ell
1}&S_{2\ell
\,2}&\cdots& S_{2\ell\,
2\ell
} \end{bmatrix}
\end{equation}
and the entries of the submatrices $S_{k,j}$ 
 $  k,j=1,\dots,2\ell$ are as  in  (\ref{t:block}).   
By arguing  as in the proof of (\ref{t:integral})  and  using (\ref{p:GIPI})  instead of (\ref{p:GIP}),  we get 
\begin{equation} \nonumber
{X}^* X-X^*SX=
 \frac{1}{h}\int_{I_{-}\cup I_{+}}\Big\| 
 \Big(  I_{2\ell}\,-\,  \Gbfia(t)\Big)
  {\widehat{X}}(t)
 \Big\|_{\C^{2\ell}}^2 
 \ dt 
\end{equation}
 for all $X\in \C^{N\,2\ell}$.
The conclusion follows by   the same arguments of the  previous theorem.
 \end{proof}
 Next we  consider the case where the   samples to  reconstruct  concern   a number  $\ng$ of generators less then the length of the space $\BL$. In Corollaries  \ref{c:K} and \ref{c:I}   below  we find  conditions  such  that  reconstruction
 is possible.   By 
renaming the  generators,  we can assume that the missing samples   are relative to the first $\ng$ generators.
\begin{corollary}\label{c:K} Assume that  the assumptions   of Theorem \ref{maingendue} are satisfied with 
$h\not= \omega/(\ell-1/2)$. Let 
 $  \cI   \subset \Z$ be   a finite set   and let $\ng$ be such that $1\le \ng <  2\ell-1$.   If  the components of index $1,2,\dots,\ng$ of the vector $W$ are 
$\PI$-independent on $K,$  
  then it is possible to recover the samples $f_j(nt_o)$,\,
   $ n\in \cI,\,  1\le j\le \ng. $
  \end{corollary}
  \begin{proof} 
Let  $  \cI=\{\ell_1,\ell_2,\dots,\ell_N\}$.
Starting  from  (\ref{t:system})  and moving to the rigth hand side  the terms containing  the known samples,     we obtain  
       \begin{align}\nonumber  
 f_k(\ell_m t_o)&-  \, 
    \sum_{j=1}^{\ng}
     \sum_{p=1 }^{N} f_j(\ell_p t_o)\,
   \big( \varphi_j^* \ast \widetilde{ \varphi}_k \big)(\ell_m t_o-\ell_p t_o) \nonumber 
     \\
   =&    \sum_{j=1}^{\ng} 
   \sum_{n\not\in  \cI }  f_j(nt_o)\,
   \big( \varphi_j^* \ast \widetilde{\varphi}_k\big)(\ell_m t_o-nt_o)   \nonumber \\
  +&    \sum_{j=\ng+1}^{2\ell-1} 
   \sum_{n\in \Z}   f_j(nt_o)\,
    \big( \varphi_j^* \ast \widetilde{ \varphi}_k \big)(\ell_m t_o-nt_o).   \nonumber 
   \end{align}  
Denoting  by $b_k(m)$ 
the right hand side we get
  \begin{equation}\label{eqlambda} 
   f_k (\ell_m t_o) -  \, 
    \sum_{j=1}^{\ng}
     \sum_{p=1 }^{N} f_j(\ell_p t_o)\,
     \big( \varphi_j^* \ast \widetilde{ \varphi}_k \big)(\ell_m t_o-\ell_p t_o) =b_k(m)
  \end{equation} 
$  1\le k\le \ng,$  $1\le m\le N.$
As in the  proof of Theorem \ref{t:recontructK} we  regard these equations as a system in the unknowns
 $ f_k(\ell_m t_o),$   $1\le m\le N, $
$  1\le k\le \ng.$
To write it in matrix form we denote the unknowns by $z_k(m)$ and  we write 
 $$ \z_k=\big(z_k(1),\ldots,z_k(N)\big) \qquad \b_k=\big(b_k(1),\ldots,b_k(N)\big)\quad 1\le k\le \ng$$
 and 
  $$Z=\big(\z_1,\ldots,\z_{\ng}\big) \qquad  B =\big(\b_1,\ldots,\b_{\ng} \big).$$
 Then the  equations  in (\ref{eqlambda})  can be written 
 \begin{equation}\label{c:small}
  \big(I -S^{[\ng]}\big)\, Z =B,
  \end{equation}
where $S^{[\ng]}$ is   the block matrix obtained from  the matrix $S $ in (\ref{t:matrix}) by eliminating the blocks $S_{j,k}$ with $j$ and $k$ greater then $\ng$,    $I$ is the $N \ng\times N\ng$ identity matrix, and    
$Z$ and $B$
 are in $\C^{N \lambda}$.   
It is possible to  recover the   missing samples   if  (\ref{c:small}) can be solved for all $B$  and this happens  if and only if the coefficient matrix $I-S^{[\ng]}$ is invertible, i.e. if and only if $1$ is not an eigenvalue of $S^{[\ng]}$.  As in the proof of  Theorem   \ref{t:recontructK}, 
 we shall show that, if 1 is an eigenvalue of  $S^{[\ng]},$ then the first $\lambda$ components of the vector $W$ are $\PI$-independent on $K.$\par
Denote by $ \Q$ the projection on $ \C^{N (2\ell-1)}$  mapping the vector $ X  = \big(\x_1 ,\ldots,\x_{2\ell-1)}\big)$ to
$   \Q
  X  = \big(\x_1,  \ldots,\x_{ \ng },0,\ldots,0\big).$ 
   We identify the operator $\Q$ with the matrix 
\begin{equation}\nonumber 
\begin{bmatrix} {\  } I_{\ng}&\vline {\ }0{\ } \\
\hline  
  {\ }0{\ }  &\vline{\ }0{\ }   \\  
   \end{bmatrix} 
\end{equation}
where $I_{\ng}$ is the $\ng\times \ng  $  identity matrix.   
We observe   that $1$ is an eigenvalue of   $S^{[\ng]}$  iff  it   is an eigenvalue of   
 $ \Q S   $ with eigenvector $X_o\in {\rm Ran}(\Q).$ 
\par  Next assume that   $1$ is an eigenvalue of the matrix
$ S^{[\ng]}$.
 Then  $1$ is an eigenvalue of $\Q
S$ and   there exists a non-null $X_o\in {\rm Ran}(\Q)$ such that  $  X_o - {{\Q }}\,S\, X_o  =0$;   i.e.  $Q X_o - {{\Q }}\,S\,Q X_o =0$. Hence 
$ {  X_o  }^*\,  Q X_o-  X_o^* \,\Q\,S\, Q X_o =0. $
 This can be written
\begin{equation}\nonumber
 {(\Q X_o)}^*  \Q X_o -   {(\Q X_o)}^*   S\, \Q X_o  =0.
\end{equation}
By applying (\ref{t:integral}) with $X=\Q X_o$  we get  
\begin{equation}\nonumber 
\big(I_{2\ell-1}\,-\,  \Gbfia  \big)(\widehat{
\Q X_o } ) =0\hskip 1truecm {\rm a.e.\,  in }\  K. 
\end{equation}
By   (\ref{p:GIP}) 
   it follows that 
  $P_W{\widehat{(\Q X_o  ^T}})=0,$
 i.e.    $ \widehat{\Q  X_o}(t)=\Q \widehat{ X_o}(t)$ is orthogonal to $W(t)$  for a.\,e. $t$ in $K.$ This shows that 
  $$\sum_{j=1}^{\ng} \overline{\what{\x_o}}_j  (t)W_j(t)  =0  \hskip 1truecm {\rm  for \ a.e.} \ t \ {\rm    in}\ K.$$
We have thus proved that   if $1$ is an eigenvalue of $S^{[\ng]},$  then 
  the first $\ng$ components of $W$ are $\PI$-dependent on  $K.$ 
   \end{proof} 
        \begin{corollary}\label{c:I} Suppose that  the assumptions   of Theorem \ref{maingen} are satisfied with $h\not= {\omega}/{\ell}$.  Let $\cI$
  be an assigned  finite set  and  $\ng$ be such that $1\le \ng <  2\ell$. If  the components of index $1,2,\dots,\ng$ of the vector $W$ are 
$\PI$-independent on  the interval $I_{-}\cup I_{+}$,  then it is possible to recover the  samples   $f_j(nt_o)$   
  $ n\in \cI ,\quad   1\le j\le \ng.$ 
 \end{corollary}
\noindent  We omit the proof  since it is quite similar  to that of  Corollary  \ref{c:K}.

 \section{Recovery of missing samples in the derivative   sampling}
\label{s:Missing}
This section is dedicated to 
derivative sampling and to 
the recovery of missing samples in the 
corresponding
reconstruction formulas.
Fix an   integer $L>1$ and let $\Phi_L=(\varphi_1,\ldots,\varphi_L)$ be defined by 
  \begin{equation}\label{Deriv:frameL} 
 \what{\varphi_1} =\chi_{[-\omega,\omega]},  \hskip 1truecm   \what{\varphi_2}=i x   \chi_{[-\omega,\omega]},\quad   \ldots,\what{\varphi}_{L}= (i x)^{L-1}    \chi_{[-\omega,\omega]}. 
  \end{equation}
 It is well known that if  $h= {2\omega}/{L}$, i.e. if $t_o={\pi L}/{
\omega}$, the family  $\EfiL$   
is a Riesz basis for $\BL$   and in the  reconstruction formula  for a function $f\in \BL$    the coefficients are  the values of $f$ and its  first $L-1$ derivatives at the sample points
(see for example \cite{HS}, \cite{Hi}).\\
In the first part of this section  we show that  $\EfiL$  is a frame  for $\BL$ for any 
$h\in \left[ {2\omega}/{L}, {2\omega}/(L-1)\right)$
 i.e. for any 
 $t_o\in \left((L-1) {\pi}/{\omega},L  {\pi}/{\omega}\right]$ (see Theorem \ref{t:derframes}).  We  need a lemma.  Denote by $M_n$ the $n\times(n+1)$ matrix 
\begin{equation}\label{t:Mmatrix}
M_n(x)= \begin{bmatrix}  
1&  \tau_{  h}(ix) &\tau_{  h}(ix)^2   &\dots   &  \tau_{  h}(ix)^{n}   
  \\
\vdots& \vdots&\vdots&{\ } & \vdots&\\
 \vdots &\vdots&\vdots&{\ }&\vdots&\\
1&  \tau_{n h}(ix) &\tau_{ nh}(ix)^2   &\dots   &  \tau_{n h}(ix)^{n}\\
\end{bmatrix} .
\end{equation}
\begin{lemma}\label{l:derjfi} Let $M_{n,r}$, $r=0,1,\dots,n$ be the submatrix of  $M_n$  obtained by suppressing the $r$-th column. Then   the  determinant  of $M_{n,r}$    is a polynomial  of degree  $n-r$.
\end{lemma}
\begin{proof}
First we observe that $ M_{n,n}$ is a Vandermonde matrix. Thus 
\begin{equation} \label{l:Vander}
{\rm det }M_{n,n}= \prod_{1\le k< j\le n } \big(\tau_{j h}(ix)-\tau_{k h}(ix) \big)=(ih)^{n }  \prod_{1\le k< j\le n }(j-k)=(ih)^{n }\prod_{p=1}^{n-1 }p!
\end{equation} 
is a constant. Next we 
  show that 
\begin{equation}\label{l:DerMmatrix}
({\rm det} M_{n,r})^\prime = i (r+1) \, {\rm det} M_{n,r+1}   \qquad r=0,\dots,n-1.
\end{equation}
We recall that  the derivative of the determinant of a $n\times n $ matrix
is the sum of $n$ determinants, each obtained by replacing  in the  matrix the elements of a column by their derivatives.
First we  prove   formula (\ref{l:DerMmatrix})    for $r=n-1$, i.e. for the matrix  
 \begin{equation}
M_{n,n-1}
=     \begin{bmatrix}  
1&  \tau_{  h}(ix) &\dots &\tau_{  h}(ix)^{n-2}      &  \tau_{  h}(ix)^{n}   
\\
\vdots& \vdots&\vdots&{\ } & \vdots&\\
 \vdots &\vdots&\vdots&{\ }&\vdots&\\
1&  \tau_{n h}(ix) &\dots &\tau_{ nh}(ix)^{n-2}      &   \tau_{n h}(ix)^{n }\\
\end{bmatrix}.  
\end{equation}
We observe that the derivative of the entries of the first column  is zero and that     for $j=1,2,\dots,n-2$    the  derivative of the $j$-th column is a multiple of the preceding column. Hence    the determinants of  the first $n-1$ matrices     
vanish.   It remains only the matrix obtained by derivating the last  column. Therefore  
\begin{equation}\nonumber
({\rm det} M_{n,n-1})^\prime = \begin{bmatrix} 
1&  \tau_{  h}(ix) &\dots &\tau_{  h}(ix)^{n-2}      & in \tau_{  h}(ix)^{n-1}   
\\
\vdots& \vdots&\vdots&{\ } & \vdots&\\
 \vdots &\vdots&\vdots&{\ }&\vdots&\\
1&  \tau_{n h}(ix) &\dots &\tau_{ nh}(ix)^{n-2}      &in  \tau_{n h}(ix)^{n-1}\\
\end{bmatrix} =\, i\, n\,  {\rm det }M_{n,n}.
\end{equation}
Hence  we have proved  formula (\ref{l:DerMmatrix}) for  $r=n-1$; 
the proof for  $r=0,1,\dots, n-2$ is similar. By (\ref{l:Vander}) and (\ref{l:DerMmatrix})    the lemma follows.
\end{proof}
\begin{theorem}\label{t:derframes} Let   $\Phi_L =(\varphi_1,\varphi_2,\ldots,\varphi_{L})$ be  as in 
 (\ref{Deriv:frameL}). If $t_o$ is  a positive number such that 
$ (L-1) {\pi}/{\omega}<t_o\le L  {\pi}/{\omega}$,
then $\EfiL$  is a frame for  $\BL$. If   $ t_o=L  {\pi}/{\omega} $ then $\EfiL$ is a Riesz basis for $\BL$.  
 \end{theorem}
\begin{proof}     We shall  prove the theorem only in the case   $L$  odd. The proof in the case     $L$ even  is similar. If  $L=2\ell-1$, then     $ {\omega}/(\ell-\frac{1}{2})\le h< {\omega}/({\ell-1})$, since $h= {2\pi}/{t_o}$.  We show  that the assumptions in Theorem  \ref{maingendue} are all satisfied. Indeed, since 
 $\sum_{j=1}^{2\ell-1}|\what{\varphi}_j|^2=\sum_{j=0}^{2\ell-1}x^{2j}
$,  condition (\ref{maingendue:1}) is satisfied. Next we prove that   there exists a positive number $\sigma$ such that   (\ref{maingendue:2}) holds. By  (\ref{colonne:1}), in the interval $K$ the matrix  obtained by deleting the vanishing row in the pre-Gramian     is    the $(2\ell-2)\times(2\ell-1)$ matrix
\begin{equation} \label{t:Jmatrix0}
 \Jbfidual(x)= \sqrt{h}\,  \begin{bmatrix} 
1& \tau_{-(\ell-1)h}(ix)
  &\dots    & \tau_{-(\ell-1)h}(ix)^{2\ell-2}   
 \\
\vdots& \vdots
&{\ } & \vdots&\\
1& \tau_{(\ell-2)h}(ix) 
 &\dots   
 & \tau_{(\ell-2)h} (ix)^{2\ell-2}\\
\end{bmatrix}
\end{equation}
i.e.
\begin{equation}\label{t:Jmatrix}
 \Jbfidual (x)=\, \sqrt{h}\  \tau_{\ell h\,} M_{2\ell-2}(x).\end{equation}
 \vskip.1truecm \noindent
By (\ref{l:Vander})    
 the $(2\ell-2)\times(2\ell-2)$ minor of $ \Jbfidual$ 
  obtained by suppressing the  last column   is   
  $$ {h}^{\ell-1}{\rm det} M_{2\ell-2,2\ell-2}(x)=
  \ (-1)^{\ell-1}  h^{3\ell-3}\, 
 \prod_{p=1}^  {2\ell-2} p! \hskip1.2truecm \forall x\in\R.$$ 
We recall that   $ \Apen{2\ell-2} $ is the sum of the squares of the absolute values of the minors of order $2\ell-2$; hence     $$\Apen{2\ell-2}\ge \, h^{6(\ell-1)}\,  \Big( \prod_{p=1}^  {2\ell-2} p!\Big)^2>0  \qquad  {\rm in }\  K.$$
This  shows that (\ref{maingendue:2}) is satisfied.
Next we consider the  interval $K_{-}$; here by  (\ref{colonne:1}) and  (\ref{Deriv:frameL})
\begin{equation}\nonumber
\Jfi =\sqrt{h}\ \ \tau_{-\ell h}\, M_{2\ell-1,2\ell-1}. 
\end{equation}
By (\ref{l:Vander})  its determinant is  
\begin{equation}\nonumber
{\rm  det} \Jfi(x)=
  (-1)^{\ell+1}\ i \, h^{3\ell-3/2}\, 
\prod_{p=1}^  {2\ell-1}
 {p!}  \hskip1truecm \forall x\in    K_{-}.
\end{equation}
The same formula holds    in    $K_{+}$.  This proves     (\ref{maingendue:3}). Thus the assumptions of     Theorem \ref{maingendue}  are satisfied.  This   concludes    the proof for the case $ {\omega}/{\ell}< h< {\omega}/(\ell-\frac{1}{2})$. 
If  $h= {\omega}/({\ell- {1}/{2}})$, then  the
  interval $K$ is empty. By arguing as before, one can  show that  (\ref{maingendue:1}) and (\ref{maingendue:3}) hold on the  intervals $K_{-}$ and $K_{+}.$ Hence, by Theorem  \ref{maingendue}, $\Efi$ is a  Riesz basis.    \end{proof} 
Denote by $ \Phi_L^*=(\varphi_1^*,\ldots,\varphi_L^*)$ the canonical dual generators of the frame $\EfiL$. By Theorem \ref{t:derframes}, every $f\in\BL$ can be expanded in terms
of the elements of the dual family as in (\ref{expansion2}).
By  
(\ref{Deriv:frameL})   the coefficients of this  sum are
$$f_j(nt_o)=\sqrt{2\pi}\   (-1)^{j-1 } f^{(j-1)}(nt_o) \qquad n\in \Z \qquad j=1,\ldots,L.$$
  Hence 
 \begin{equation}\label{DSformulaL}
f(x)= {\sqrt{2\pi}}\sum_{j=1}^L\sum_{n\in\Z} 
(-1)^{j-1} \, f^{(j-1)}(nt_o)\   \varphi_j^*(x-nt_o) \qquad \forall f\in\BL,
\end{equation}
where $L=2\ell\, $ if $\, {\omega}/{\ell}\le h< {\omega}/(\ell- {1}/{2})$  and $L=2\ell-1\,$ if $ \,{\omega}/(\ell- {1}/{2})\le h<  {\omega}/(\ell-1).$
We shall call (\ref{DSformulaL})   the {\it      derivative oversampling   formula  of order $ L-1$}.  
\par
     Proposition \ref{p:recder}  below shows that, for band-limited signals,   it is possible to recover any  finite set of  missing samples  in the derivative sampling expansion (\ref{DSformulaL}). Thus we generalize to arbitrary order the result of   
  Santos and  Ferreira   in  \cite{SF}.
We need a lemma. Denote 
   by $\partial P$ the degree of a polynomial  $P.$   
\begin{lemma}\label{l:dertre} Let  $P_i$,  $i=1,2,\dots,N$, be  
non-null polynomials such that 
 $$\partial P_1 >  \partial P_{2} >\ldots \ldots >  \partial P_N. $$
 If ${ \sum_{j=1}^N \gamma_j P_j=0}$  for some continuous  $h$-periodic functions $\gamma_i$, $i=1,2,\dots,N$,  
 then   $\gamma_i= 0$ for all $i=1,2,\dots,N.$
\end{lemma}
   \begin{proof}
   The case $N=1$ is obvious; assume  by induction that the statement  is true for  $N$  polynomials. Let  $P_i$,  $i=1,2,\dots,N+1,$ be  
non-null polynomials such that 
 $$\partial P_{ 1} >  \partial P_{2} >\ldots \ldots >  \partial P_{N }>\partial P_{N+1}$$
   and 
   $$ \gamma_{1} P_{1}+ \gamma_{2} P_{2} +\ldots \ldots +   \gamma_{N} P_{N}+ \gamma_{N+1} P_{N+1}=0.$$   Denote  by $q$ be the  degree of $P_{1}$. By  dividing by $x^q$ and  taking the limit  to infinity, we  obtain  that $\gamma_{1} $ is   identically zero. Hence ${ \sum_{j=2}^{N+1} \gamma_j P_j=0}$. The conclusion follows, by the inductive assumption.   \end{proof}
 \begin{proposition}\label{p:recder}  Let $\Phi_L =(\varphi_1,\varphi_2,\ldots,\varphi_{L})$ the frame generators defined by    
(\ref{Deriv:frameL}). If  
$ (L-1) {\pi}/{\omega}<t_o< L \ {\pi}/{\omega},$
then in the  sampling  formula (\ref{DSformulaL}) it is possible to recover the  missing samples 
$$f^{(j)}(nt_o) 
\hskip.5 truecm n\in \cI,\   1\le j\le L-1,  $$  
for any finite set   $  \cI\subset\Z$.  
\end{proposition}  
   \begin{proof}
  We shall only prove the theorem when  $L$ is odd. The proof for $L$ even is similar. Let $L=2\ell-1;$ then $ {\omega}/(\ell- {1}/{2})\le h< {\omega}/(\ell-1)$, by the assumption   on $t_o$.  
   We recall from Section \ref{s: Recovery} that $W$ denotes the cross product  of the rows of $ {h}^{-\frac{1}{2}} \Jbfi$; thus, up to a sign,  its    components $W_k$  are  the minors  of order $2\ell-2$ of $ {h}^{-\frac{1}{2}} \Jbfi$. 
  By Theorem  \ref{t:recontructK}   it suffices to prove that the components of the vector $W$   are $\PI$-independent on $K$ for any finite set   $ \cI.$ 
By (\ref{t:Jmatrix})  and Lemma \ref{l:derjfi},  $W_k$ is a polynomial of degree $ 2\ell-1-k$      for $ k=1,2,\dots 2\ell-1.$
Thus 
 $$\partial W_{ 1} >  \partial W_{2} >\ldots \ldots  >\partial W_{2\ell-1}.$$
Let now $ {p}_1,p_2,\ldots,p_{2\ell-1}$ be trigonometric polynomials in $\PI$ such that 
$$\sum_{j=1}^{2\ell-1}\, p_j(t)\, W_j(t)=0,\qquad t\in K.$$
By analytic continuation, this identity holds on $\R.$ Thus, by Lemma  \ref{l:dertre}, $p_j=0$ for all $j=1,\ldots,2\ell-1$. This concludes the proof of  the proposition.   \end{proof}  
  \section{The  first order formula: explicit duals and numerical experiments }
\label{s:firstorder}
 Obviously  in applications,   to  use  the recontruction formula (\ref{DSformulaL}),    it is necessary to find     the dual generators   $  {\varphi_1}^*,\dots, {\varphi_L^*}$. 
When $L=2, 3$ we gave an explicit expression of the  Fourier transforms of the duals  in terms of the Fourier transforms of the generators   in(see Corollaries 5.2 and 5.3 in   \cite{DP}). For greater   values of $L$ one can  use the general formulas given in Theorems 4.6 and 4.7 in \cite{DP}. In this section we provide   an explicit espression of the duals for the first-order derivative formula ($L=2$) (see (\ref{dsamp:formdue})) and  
  we    present  a numerical experiment  of  recovery of  missing samples   and   reconstruction of a signal.  
  \par 
 Let    $ {\varphi_1}, {\varphi_2}$ be the functions defined in   (\ref{Deriv:frameL})  for $L=2:$ 
 \begin{equation}\nonumber
 \what{\varphi_1}=\chi_{[-\omega,\omega]}  \hskip 1truecm   \what{\varphi_2}=i x   \chi_{[-\omega,\omega]} . 
\end{equation}
  By Theorem   \ref{t:derframes},     $\Efi$ is a   frame for $\BL$ if $\omega\le h<2\omega$; if $h=\omega$ then it is a Riesz basis for $\BL.$  The sampling formula (\ref{DSformulaL}) reduces to    \begin{equation}\label{dsamp:expdue}
f=\sqrt{2\pi}\sum_{k\in \Z}   f(kt_o)\tau_{-kt_o}\varphi_1^*- f'(kt_o) \tau_{-kt_o}\varphi_2^* .
\end{equation}
The Fourier transforms of the dual generators  are  \begin{equation} \nonumber 
\what{\varphi_1^*}(x)=
\begin{cases}
\frac{1}{h}(1-  \frac{|x|}{h})  &\hskip-.1truecm  h- \omega<|x|< \omega\\
   \frac{1}{h (1+x^2)} &\hskip-.1truecm|x|<h- \omega  \\
\end{cases}
\hskip 0.4truecm \what{\varphi_2^*}(x)=
\begin{cases}
\frac{i}{h^2} \,{\rm sign}(x)  &\hskip-.1truecm h- \omega<|x|< \omega\\
  \frac{ix}{h (1+x^2)}  &\hskip-.1truecm |x|<h- \omega  \\
  \end{cases}
\end{equation}  (see  Corollary 5.2 in \cite{DP}).   Figure 1   shows $\what{\varphi_1 ^*},\ \what{\varphi_2 ^*}$  for  $\  h=1.6\,   \omega$.\
 \begin{figure}[h]
 \begin{center}
\centering\setlength{\captionmargin}{0pt}
  \includegraphics[width=12.5cm,height=2.8cm]{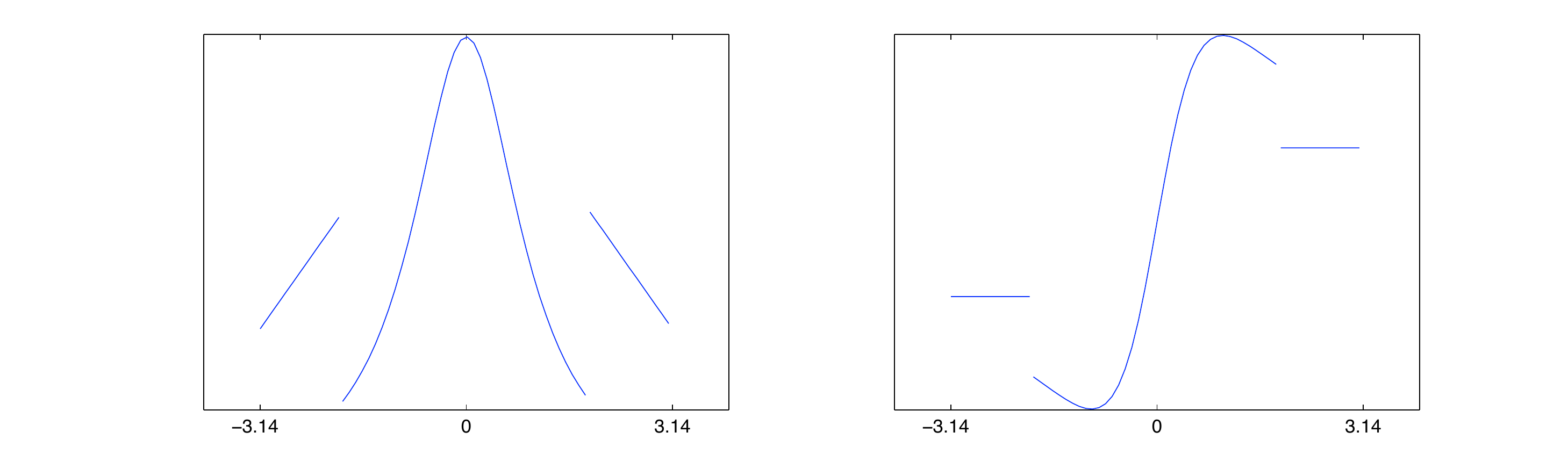}
 \vskip-.2truecm
 \caption{The Fourier transforms of the duals  for $h=1.6\ 
  \pi$.}  \end{center}
\end{figure}  
 \par \noindent
 By  Fourier inversion
  \begin{equation}\label{dsamp:formdue}
 \begin{aligned}
\varphi_1^*(x)=& \frac{1}{\sqrt{2\pi}}\Big[\,  \frac{2\omega}{h} H \big(\snc(\omega x) -\snc(Hx)\big)+\frac{2} {h^2x^2}\big(\cos(Hx) -\cos(\omega x)\big)\\
&+\frac{1}{h} \int_{|t|<H}\, \frac{1}{1+t^2} e^{itx}\,dt\,\Big]
  \\
{\ }  
\\
\varphi_2^*(x)=&\frac{1}{\sqrt{2\pi}}\Big[\, \frac{2}{h^2}\big(\cos(\omega x) -\cos(Hx)\big)x^{-1} +\frac{i}{h}\int_{0<|t|<H} \frac{t}{1+t^2} e^{itx}\, dt \Big]
\end{aligned}\end{equation}
where $H=h-\omega$ and $\snc={\sin(x)}/{x}$. 
.  The integrals in the right hand side of this formula  can be written as linear combinations of sine  and cosine integrals, so that  it is possible to implement  the duals  without any approximation. 
Observe  that if $h=\omega$ then
$$
\varphi_1^*(x)=\frac{1}{ \sqrt{2 \pi}} \snc^2\left(\frac{\omega x}{2}\right) \qquad \varphi_2^*(x)=- \frac{1}{ \sqrt{2 \pi}}\,  x\,  \snc^2\left(\frac{\omega x}{2}\right).
$$ 
  Following Ferreira,    we have chosen  the    test function    \begin{equation}\label{ferreira}
\fo (x)=\snc(\pi(x-2.1))-  0.7\,  \snc(\pi(x+1.7)) 
\end{equation}
plotted in   Figure 2.
 \begin{figure}[h]
\begin{center}
\centering\setlength{\captionmargin}{0pt}%
\includegraphics[width=6.3cm,height=4cm]
{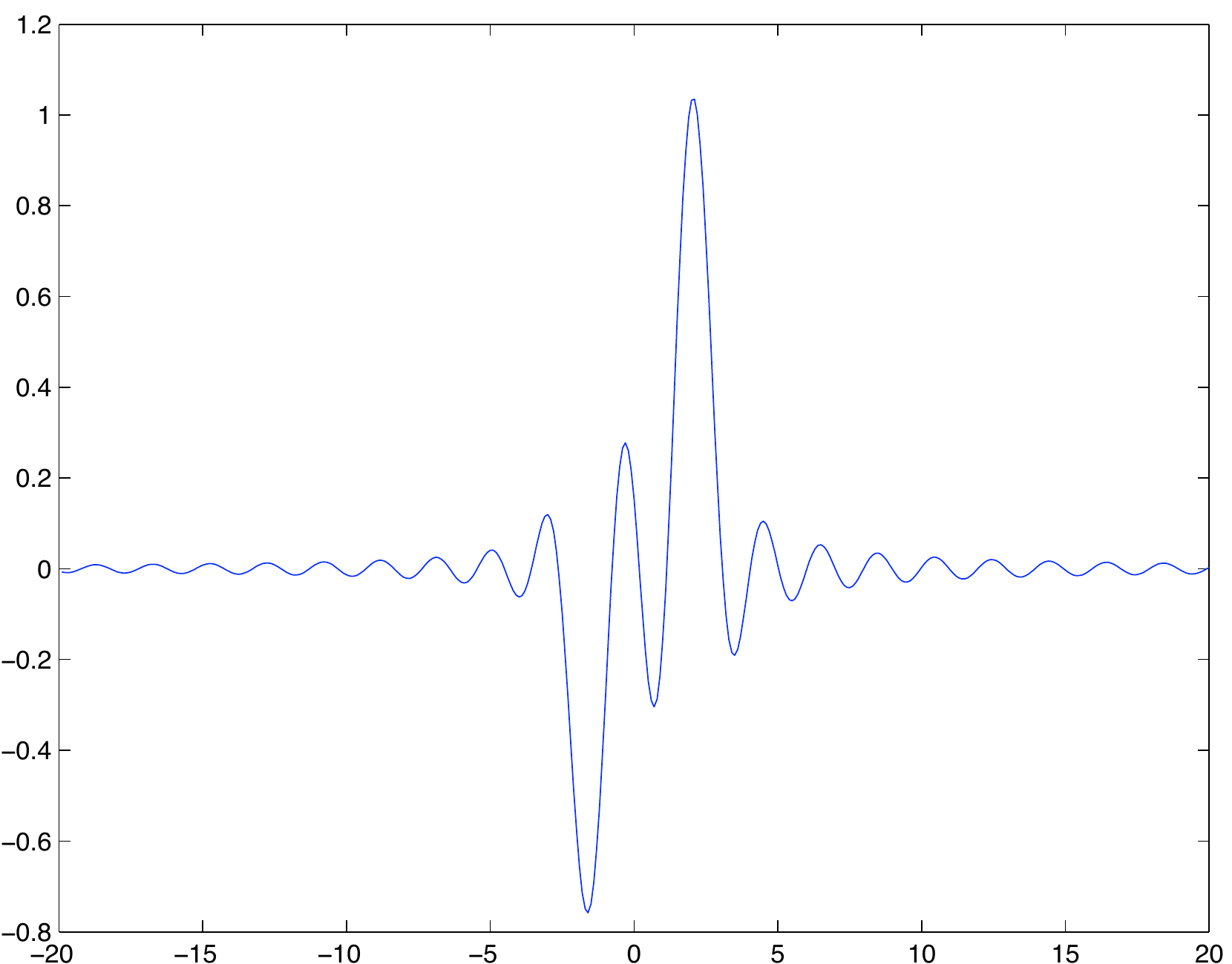} 
 \caption{The function $ \fo$ in $[-20,20].$}
  \end{center}
\end{figure}
  In general, redundancy  makes frame  expansions robust to errors like noise,  quantization and losses. However, for this  derivative frame, even the  reconstruction
  of a signal  where only two samples are   missing may  lead to  poor results. To illustrate  this point,  we  have    reconstructed    the  signal   $ \fo$ 
   in (\ref{ferreira}) 
by using   formula  
 (\ref{dsamp:expdue}), when  the samples $ \fo(4t_o)$ and  $\fo'(4t_o)$  are set to zero: 
 \begin{equation}\nonumber 
 \fo^*=\sqrt{2\pi}\sum_{k\not=4}   \fo(kt_o)\tau_{-kt_o}\varphi_1^*-  \fo'(kt_o) \tau_{-kt_o}\varphi_2^* .
\end{equation}
 Figure 3 shows the  functions $ \fo^*$ and  $ \fo$   in the interval $[-10,10]$. 
  \begin{figure}[h]
\begin{center}
\centering\setlength{\captionmargin}{0pt} 
\includegraphics[width=6.5cm,height=4cm]  {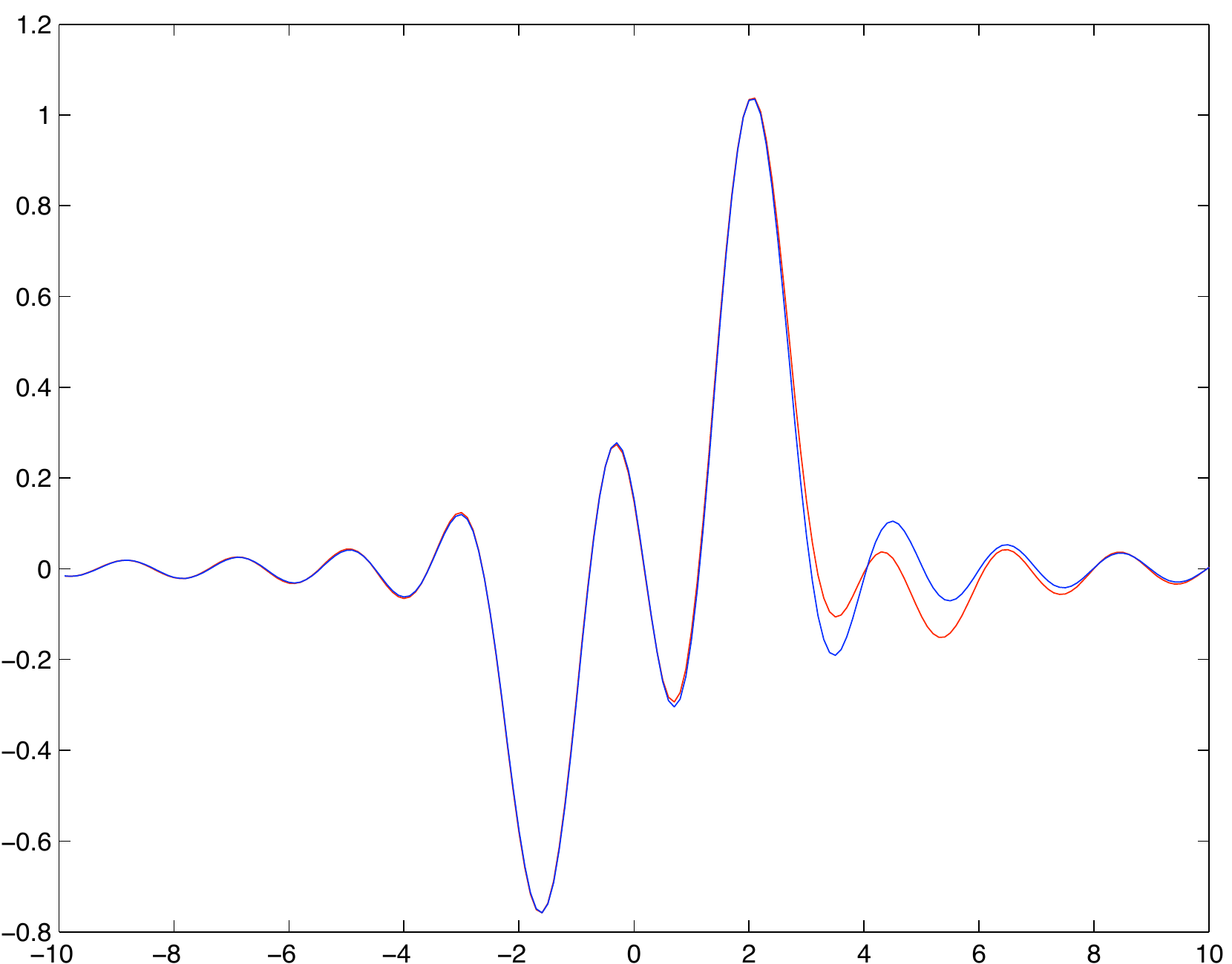}  \caption{The functions $ \fo^*$  and  $ \fo$  in $[-10,10]$.    } \end{center}
\end{figure}
\par \noindent  In our experiment we assume that    $10$  samples of  $\fo$ and $\fo'$ are unknown and  we recover them  by solving the $20\times 20$  system (\ref{t:matrixeq}).  Then we obtain the reconstructed function  $\fo^c$  via   formula (\ref{dsamp:expdue}), where  the missing 
   samples are replaced by the  computed ones.  
We suppose  that the missing samples are 
$ \{\fo(\ell_k t_o), \fo'(\ell_k t_o),$\penalty-10000 $  k=1,2,\dots,10\}$, where  $ l_k=-16+3(k-1)$.    We  have obtained  the   computed samples with   an   error    of  order $10^{-4}$ and  a   relative error  of   order $10^{-2} $; the    graph  of the reconstructed function $\fo^c$ is undistinguishable from that of  $\fo$.
     
   \end{document}